\documentclass[10pt]{article}

\author{Patrick Dondl, Martin Jesenko}
\title{Threshold phenomenon for homogenized fronts in random elastic media}
\date{Freiburg, 2019}

\usepackage[a4paper, twoside , hcentering , bindingoffset = 1 cm , scale={0.8,0.8} , vmarginratio=11:11 , headsep=0.8cm]{geometry}
\usepackage{pifont,stmaryrd}
\usepackage[pdftex]{hyperref}
\usepackage{fancyhdr}
\usepackage{amssymb, amsmath, amstext, amsthm}
\usepackage{enumerate}
\usepackage{eucal}
\usepackage{mathrsfs}
\usepackage{aurical}
\usepackage[T1]{fontenc}

\usepackage{pictexwd,dcpic}
\usepackage{graphicx}
\theoremstyle{plain}
\newtheorem{theo}{Theorem}[section]

\newtheorem{prop}[theo]{Proposition}

\theoremstyle{definition}
\newtheorem{ass}[theo]{Assumption}

\theoremstyle{remark}

\def\ds {\displaystyle}

\def\prostor{\, \underline{\phantom{x}} \,}

\newcommand{\Laplace}{\Delta}

\def\x {\ dx}
\def\y {\ dy}

\def\R {\mathbb R}

\def\N {\mathbb N}
\def\Z {\mathbb Z}

\def\i {\infty}

\def\* {$*$}
\def\eps {\varepsilon}

\def\Re{\mathop{\rm Re}}

\usepackage{color}

\begin{document}
\maketitle

\section{Introduction}
We are investigating the behaviour of solutions $u\colon \R^n\times [0,\infty) \to \R$ of the fractional semilinear partial differential equation
\begin{eqnarray} 
- ( - \Laplace )^{s} u(x,t) - f( x , u(x,t) ) + F & = & \partial_{t} u(x,t) , \label{eq:model} \\
u(x,0) & = & U(x), \nonumber
\end{eqnarray}
for a given suitable initial condition $ U \in C^{2}( \R^{n} ) $. 
As usual, by $  ( - \Laplace )^{s} $ we denote the fractional Laplacian operator (in space variables):
If a function $v$ is such that $ v - \ell $ is sublinear for some linear function $ \ell $, then (in the spirit of e.g.~\cite{MonneauPatrizi1})
\[ - ( - \Laplace )^{s} v(x)
:= \lim_{ r \searrow 0 } \int_{ B_{1/r}(0) \setminus B_{r}(0) } \frac{ v(x+h) - v(x) }{ |h|^{n + 2s} } \ dh. \]

This equation models the evolution of an interface given by the graph $(x,u(x,t))$ of the function $u$ in a heterogeneous environment and with nonlocal interaction. The heterogeneity of the medium is given by the function $f$, which is evaluated at the interface, i.e., this interaction is assumed to be local. We add a constant, external driving force $F\ge 0$. Our main interest lies in the case where the nonlinearity $f$ is given by a random distribution of localized obstacles, to be specified more precisely later.

In this setting, we construct a stationary supersolution, i.e., a stationary (random) function $v$ such that
\begin{eqnarray*} 
- ( - \Laplace )^{s} v(x,t) - f( x , v(x,t) ) + F^* & \le & 0 \\
v(x) & > & U(x).
\end{eqnarray*}
for suitable $U$ and positive, but sufficiently small $F^*$. We note that---due to the comparison principle for the evolution equation---this yields that
\[
u(x,t) \le v(x) \quad \text{for all $t\ge 0$},
\]
i.e., the function $v$ acts as a barrier for propagating interfaces, thus yielding a hysteresis effect.

We note that equations of the form of~\eqref{eq:model}, especially for $s=\frac{1}{2}$ arise in a large number of physical systems. This is due to the fact that the half-Laplacian $- ( - \Laplace )^{\frac{1}{2}}$ arises as the variation of the square of the $H^{\frac{1}{2}}$-seminorm, which is nothing but the Dirichlet energy of a harmonic function on an extension domain with given boundary condition.

Some notable examples have been discussed in~\cite{DondlScheutzowThrom, Throm} (where the problem has been investigated for the case $n=1$) and include the propagation of a wetting line on a rough surface~\cite{wetting} and the propagation of crack fronts in rough media~\cite{schmittb1, schmittb2}. In higher space dimensions, eqation~\eqref{eq:model} describes (in a shallow interface limit), the propagation of twin boundaries in elastic solids~\cite{DondlBhattacharya}. The understanding of pinning, i.e., the existance of suitable stationary supersolutions to the evolution equation is essential for the understanding of precipitate hardening in TWIP steel~\cite{courtedonbhat}.

The remainder of the article is organized as follows:
%

After introducing the setting, we find in Section~\ref{sect:LFS} a local flat solution. 
More precisely, we look for a solution of the Dirichlet problem for the fractional Laplacian on a ball.
The solution is radially increasing and has the appropriate fractional Laplacian in the interior of a ball
with zero values on its complement.

In Section~\ref{sect:decomposition} we decompose the half space above the zero plane 
and thus control the probabilities of existence of (enough strong) obstacles.
By a appropriate transformation this also applies to the case where the initial plane is tilted.
We define a flat solution simply as a minimum of local flat solutions.

Section~\ref{sect:percolation} is devoted to a percolation result that supplies sublinearly increasing function.

Next we lift the local flat solution accordingly to the positions of obstacles. 
For the non-zero initial values, we just add the initial function $ x \mapsto U(x) $.

In Section~\ref{sect:scaling} we summarize all the conditions and see that they can be met.
We clearly state our main result in Theorem~\ref{theo:main-finite}.

As a consequence, for the case $ s = \frac{1}{2} $ we gain an insight into the behaviour if we shrink the setting by some $ \eps \ll 1 $. This homogenization result in the pinning regime is stated in Theorem~\ref{theo:homogenization}.

\section{Setting}
Now let us precisely state assumptions on our random field of obstacles.
We suppose that the obstacles have the same shape and random positions and strenghts
 and that the obstacles do not depend on time, i.e., we are in a quenched setting.
The force of the obstacle field is a random function
\[ f : \R^{n} \times \R \times \Omega \to \R, 
\quad
f( x , y, \omega ) := \sum_{i=1}^{\i} f_{i}( \omega ) \varphi( x - x_{i}( \omega ), y - y_{i}( \omega ) ), \]
where $\Omega$ is a probability space. The function $f$ is assumed to satisfy the following hypothesis.
\begin{ass}
\label{ass:1}
\begin{enumerate}
\item[]
\item 
Shape of obstacles: Function $ \varphi $ belongs to $ C_{c}^{\i}( \R^{n} \times \R ) $ and satisfies
\[ \varphi(x,y) \ge 1 \mbox{ for } \| (x, y) \|_{\i} \le r_{0}
\quad \mbox{and} \quad
\varphi(x,y) = 0 \mbox{ for } \| (x, y) \| \ge r_{1} \]
for some $ r_{0}, r_{1} > 0 $ with $ r_{1} > \sqrt{n} r_{0} $.
\item 
Obstacle positions: $ \{ ( x_{i} , y_{i} ) \}_{ i \in \N } $ are distributed 
according to an $ (n + 1) $-dimensional Poisson point process on $ \R^{n} \times [ r_{1} , \i ) $ 
with intensity $ \lambda > 0 $.
\item 
Obstacle strengths: $ \{ f_{i} \}_{ i \in \N } $ are independent and identically distributed strictly positive random variables
($ f_{i} \sim f_{0} $ for all $ i \in \N $) that are independent of $ \{ ( x_{i} , y_{i} ) \}_{ i \in \N } $ .
\end{enumerate}
\end{ass}
%
As already mentioned this problem was already solved for $ n = 1 $ and for the zero initial value:
\begin{theo}[\cite{DondlScheutzowThrom}]
Suppose that $ n = 1 $ and that Assumption~\ref{ass:1}~is satisfied.
There then exist a deterministic $ F_{*} > 0 $ 
and a continuous random function $ v : \R \times \Omega \to [0, \i) $ with the property that the function
$ \overline{v}( x, t, \omega ) := \min \{ F_{*}t , v(x,\omega) \} $ is a viscousity supersolution to the
evolution problem 
\begin{eqnarray*} 
- ( - \Laplace )^{s} u(x,t,\omega) - f( x , u(x,t,\omega) , \omega ) + F & = & \partial_{t} u(x,t,\omega),  \\
u(x,0,\omega) & = & 0,
\end{eqnarray*}
for $ F \le F_{*} $ and for almost every $ \omega \in \Omega $.
Furthermore, we can choose $v$ such that there exist constants $ C > 0 $ and $ q > 0 $
so that for any $ x \in \R $ we have $ \mathbf{P}( v(x) > h ) \le C e^{-qh}$, i.e.~the height of the pinned
interface admits an exponential tail in its distribution. In particular, for any $ x \in \R $,
the expected value of the height of the pinned interface satisfies $ \mathbf{E}( v(x) ) < \beta $ for
some fixed $ \beta < \i $, depending only on the deterministic parameters of the obstacle
distribution and on $s$.
\end{theo}
Dealing with a non-local operator, the authors choose to use periodic functions.
We will not generalize their approach to higher dimensions here 
since it seems technically difficult to show the analogous monotonicity properties and understand intersections of local solutions.
Our approach is in fact closer to the one in the proof of analogous result in the local setting.
\begin{theo}[\cite{DirrDondlScheutzow}]
If Assumption~\ref{ass:1}~is satisfied, 
then there exists $ F_{*} > 0 $ and a non-negative $ v : \R^{n} \times \Omega \to [ 0 , \i ) $
so that 
\[ 0 \ge \Laplace v(x, \omega) - f( x, v(x,\omega), \omega ) + F_{*} \]
almost surely. 
\end{theo}
Again, due to the comparison principle, such a supersolution blocks any propagating solution that starts below, and exponential tail estimates hold.
Cleary, the fact that here were are dealing with a non-local operator requires a special attention
starting already by introducing an appropriate concept of the Dirichlet problem.

\section{Local flat supersolution} \label{sect:LFS}
Following the idea in \cite{DirrDondlScheutzow}, 
we first construct a local supersolution inside some ball.
We suppose that in the center of this ball, there lies an obstacle of a sufficient strength.
The remaining part of this ball may or may not contain any other obstacle.
The main difference is that here we are dealing with the fractional Laplacian, which is a non-local operator. 
A way of formulating the Dirchlet problem on some open subset of $ \R^{n} $ is 
to prescribe values not only on the boundary but on the entire complement. 
Let us mention that there exists also an alternative non-equivalent notion of the Dirichlet problem 
via spectrum, see e.g.~\cite{ServadeiValdinoci}.

Due to the properties of the fractional Laplacian, we may restrict ourselves to balls centered at the origin.
Our task is to find a function $u$ such that 
\[ - ( - \Laplace )^{s} u(x) \le f( x , u(x) ) - F. \]
Hence, $ - ( - \Laplace )^{s} u $ may be positive inside an obstacle, 
i.e.~in a small concentric ball with radius $ r_{0} $,
and must be negative on its complement.
Therefore, let us first for some $ R > r_{0} $ and $ F_{1} , F_{2} $ constract a radial solution to
\begin{eqnarray*} 
- ( - \Laplace )^{s} u(x) & = &  
\left\{ \begin{array}{cl} 
F_{1}, & \mbox{if } |x| < r_{0}, \\
- F_{2}, & \mbox{if } r_{0} < |x| < R, 
\end{array} \right. \\
u(x) & = & 0 \quad \mbox{if } |x| \ge R. 
\end{eqnarray*}
%

Green's function for the ball $ B_{R}(0) \subset \R^{n} $, $ n \ge 2 $, i.e.~the distributional solution of
\[ - ( - \Laplace )^{s} u(x) + \delta_{y}(x) = 0 \quad \forall x \in B_{R}(0) \]
with $ u(x) = 0 $ for $ x \not\in B_{R}(0) $, is
\[ G_{n,s}(x,y) 
= \frac{ \Gamma( \frac{n}{2} ) }{ 2^{2s} \pi^{1/n} \Gamma(s)^{2} }
	\frac{1}{ |x-y|^{n-2s}  } \Phi_{n,s}(x,y)
\]
where
\[ \Phi_{n,s}(x,y) = \int_{0}^{ \zeta } \frac{1}{ w^{1-s} } \frac{1}{ (1+w)^{n/2} } \ dw \]
and
\[ \zeta = \frac{1}{ R^{2} } \frac{ ( R^{2} - |x|^{2} ) ( R^{2} - |y|^{2} ) }{ |x-y|^{2} } \]
(see~\cite{Bucur}, Theorem 3.1, or \cite{Pozrikidis}, p.~249--250).
Using the Euler type integral expression for the hypergeometric function $ {}_{2} F_{1} $
for $ \Re c > \Re b > 0 $ and $ z \not\in [ 1 , \i ) $
\[ B( b , c-b ) \cdot {}_{2} F_{1}(a,b,c;z) = \int_{0}^{1} x^{b-1} ( 1 - x )^{c-b-1} ( 1 - z x )^{-a} \x, \]
where $B$ is the beta function, we may rewrite
\begin{eqnarray*}
\int_{0}^{ \zeta } \frac{1}{ w^{1-s} } \frac{1}{ (1+w)^{n/2} } \ dw
& = & \int_{0}^{1} \frac{1}{ ( \zeta t )^{1-s} } \frac{1}{ ( 1 + \zeta t )^{n/2} } \ \zeta \ dt \\
& = & \zeta^{s} B( s , 1 ) \cdot {}_{2} F_{1}( \tfrac{n}{2} , s , s+1 ; - \zeta ) \\
& = & \frac{ \zeta^{s} }{s} \cdot {}_{2} F_{1}( \tfrac{n}{2} , s , s+1 ; - \zeta ).
\end{eqnarray*}
For given $ F_{1} , F_{2} > 0 $, the solution of
\begin{eqnarray*} 
- ( - \Laplace )^{s} u(x) & = &  
\left\{ \begin{array}{cl} 
F_{1}, & \mbox{if } |x| < r_{0}, \\
- F_{2}, & \mbox{if } r_{0} < |x| < R, 
\end{array} \right. \\
u(x) & = & 0 \quad \mbox{if } |x| \ge R, 
\end{eqnarray*}
exists and lies in $ C^{0,s}( \R^{n} ) $ (Proposition 1.1 in \cite{ROS}).
It is given by
\[ u(x) 
= F_{2} \int_{ B_{R}(0) } G_{n,s}(x,y) \y 
- ( F_{1} + F_{2} ) \int_{ B_{ r_{0} }(0) } G_{n,s}(x,y) \y. \]
Namely, although our source is not smooth in $ B_{R}(0) $, 
it can be approximated from below and from above with smooth functions.
For them, the solution may be computed with Green's function (e.g.~Theorem 3.2 in \cite{Bucur}).
Then we apply the comparison principle, see \cite{RO-1} and the references therein.

Let $ r_{0} = q R $ with $ q \in (0,1) $.
We would like to explore the interplay of these two integrals 
and find an appropriate scaling for $ F_{1}, F_{2}, q $ 
in order for the solution to be non-positive and monotonically increasing away from the origin.

The first integral is known: It is the solution for an uniform source (compare \cite{Getoor})
\[ g(x)
:= \int_{ B_{ R }(0) } G_{n,s}(x,y) \y 
= \frac{ \Gamma( \frac{n}{2} ) }{ 2^{2s} \Gamma( \frac{n}{2} + s ) \Gamma( 1+s ) } ( R^{2} - |x|^{2} )^{s}. \]
Let us denote
\[ b(x) := \int_{ B_{ r_{0} }(0) } G_{n,s}(x,y) \y. \]
From the assessments
\[ \int_{0}^{ \zeta } \frac{1}{ w^{1-s} } \frac{1}{ (1+w)^{n/2} } \ dw 
\ge  \int_{0}^{ \zeta } \frac{1}{ w^{1-s} } \frac{1}{ (1 + \zeta )^{n/2} } \ dw 
= \frac{ \zeta^{s}}{ (1 + \zeta )^{n/2} s } \]
and 
\begin{eqnarray*}
( 1 + \zeta ) |x-y|^{2}
&  =  & |x-y|^{2} + \frac{ ( R^{2} - |x|^{2} ) ( R^{2} - |y|^{2} ) }{ R^{2} } \\
&  =  & |x|^{2} - 2 x \cdot y + |y|^{2} + R^{2} - |x|^{2} - |y|^{2} + \frac{ |x|^{2} |y|^{2} }{ R^{2} } \\
&  =  & - 2 x \cdot y + R^{2} + \frac{ |x|^{2} |y|^{2} }{ R^{2} } \\
& \le & \frac{ 2 R^{2} |x| |y| + R^{4} + |x|^{2} |y|^{2} }{ R^{2} } \\
&  =  & \frac{ ( R^{2} + |x| |y| )^{2} }{ R^{2} },
\end{eqnarray*} 
it follows
\begin{eqnarray*}
\frac{ \Phi_{n,s}(x,y) }{ |x-y|^{n-2s}  }
& \ge & \frac{ \zeta^{s}}{ (1 + \zeta )^{n/2} s |x-y|^{n-2s} } \\
&  =  & \frac{ ( \zeta |x-y|^{2} )^{s} }{ s ( ( 1 + \zeta ) |x-y|^{2} )^{n/2} } \\
& \ge & \frac{ \left( \frac{ ( R^{2} - |x|^{2} ) ( R^{2} - |y|^{2} ) }{ R^{2} } \right)^{s} }
		{ s  \left( \frac{ ( R^{2} + |x| |y| )^{2} }{ R^{2} } \right)^{n/2}  }  \\
&  =  & \frac{  R^{n-2s} }{s} \frac{ ( R^{2} - |x|^{2} )^{s} ( R^{2} - |y|^{2} )^{s} }{ ( R^{2} + |x| |y| )^{n} }.  
\end{eqnarray*} 
Thus, for $ |y| < r_{0} $
\[ \frac{ \Phi_{n,s}(x,y) }{ |x-y|^{n-2s}  }
\ge \frac{  R^{n-2s} }{s}  \frac{ ( R^{2} - |x|^{2} )^{s} ( R^{2} - r_{0}^{2} )^{s} }{ ( R^{2} + R r_{0} )^{n} }
\ge \frac{1}{s} \frac{ ( R^{2} - |x|^{2} )^{s} ( 1 - q^{2} )^{s} } { ( 1 + q )^{n} R^{n} }. \]
Hence,
\begin{eqnarray*}
\int_{ B_{ q R }(0) } G_{n,s}(x,y) \y
& \ge & \frac{ \Gamma( \frac{n}{2} ) }{ 2^{2s} \pi^{1/n} \Gamma(s)^{2} s }
		\frac{ ( 1 - q^{2} )^{s} }{ ( 1 + q )^{n} R^{n}  } 
		| B_{ q R }(0) | ( R^{2} - |x|^{2} )^{s}  \\
&  =  & \frac{ \Gamma( \frac{n}{2} ) }{ 2^{2s} \pi^{1/n} \Gamma(s)^{2} s }
		\frac{ ( 1 - q^{2} )^{s} }{ ( 1 + q )^{n} } 
		q^{n} \frac{ \pi^{n/2} }{ \Gamma( \frac{n}{2} + 1 ) }( R^{2} - |x|^{2} )^{s}  \\	
&  =  & \frac{ ( 1 - q^{2} )^{s} }{ ( 1 + q )^{n} } q^{n}
		\cdot \frac{ \pi^{ n/2 - 1/n } \Gamma( \frac{n}{2} + s ) }{ \Gamma( \frac{n}{2} + 1 ) \Gamma(s) }
		\cdot \frac{ \Gamma( \frac{n}{2} ) }{ 2^{2s} \Gamma( \frac{n}{2} + s ) \Gamma( s + 1 ) } ( R^{2} - |x|^{2} )^{s}. 
\end{eqnarray*}
Since we are considering the case $ n \ne 1 $,
\[ \frac{ \pi^{ n/2 - 1/n } \Gamma( \frac{n}{2} + s ) }{ \Gamma( \frac{n}{2} + 1 ) \Gamma(s) }
= \frac{ \pi^{ n/2 - 1/n } }{ \frac{n}{2} B( \frac{n}{2} , s ) }
\ge \frac{2s}{n}. \]
%
Therefore,
\[ b(x) \ge \frac{ ( 1 - q^{2} )^{s} }{ ( 1 + q )^{n} } q^{n} \frac{2s}{n} g(x). \]
Hence,
if
\[ \frac{ F_{1} + F_{2} }{ F_{2} } \ge \frac{n}{2s} \frac{ ( 1 + q )^{n} }{ ( 1 - q^{2} )^{s} } \frac{1}{ q^{n} }, \]
then $ u(x) < 0 $ for every $ x \in B_{R}(0) $. 

Now, we would also like $u$ to be monotonically increasing (in the radial direction away from the origin). 
We already know that
\[ \nabla g(x)
= \frac{ \Gamma( \frac{n}{2} ) }{ 2^{2s} \Gamma( \frac{n}{2} + s ) \Gamma( 1+s ) } ( -2 s x )( R^{2} - |x|^{2} )^{s-1}. \]
Therefore, let us look at
\[ - \nabla b(x) 
= \int_{ B_{ q R }(0) } - \nabla_{x} G_{n,s}(x,y) \y. \]
First,
\begin{eqnarray*}
\nabla_{x} \frac{ \Phi_{n,s}(x,y) }{ |x-y|^{n-2s}  }
&  =  & ( 2 s - n ) \frac{ x - y }{ |x-y|^{n-2s+2} } \Phi_{n,s}(x,y)
		+ \frac{1}{ |x-y|^{n-2s} } \frac{1}{ \zeta^{1-s} } \frac{1}{ (1+\zeta)^{n/2} } \nabla_{x} \zeta \\
&  =  & ( 2 s - n ) \frac{ x - y }{ |x-y|^{2} } \frac{ \Phi_{n,s}(x,y) }{ |x-y|^{n-2s}  }
		+ \frac{1}{ |x-y|^{n-2s} } \frac{1}{ \zeta^{1-s} } \frac{1}{ (1+\zeta)^{n/2} } \nabla_{x} \zeta.	
\end{eqnarray*} 
Since
\[ \nabla_{x} \zeta
= \frac{ R^{2} - |y|^{2} }{ R^{2} } 
\left( \frac{ - 2 x }{ | x - y |^{2} } - ( R^{2} - |x|^{2} ) \cdot 2 \frac{ x - y }{ | x - y |^{4} } \right) 
\]
and
\[ 1 + \zeta
= 1 + \frac{ ( R^{2} - |x|^{2} ) ( R^{2} - |y|^{2} ) }{ R^{2} |x-y|^{2} }
= \frac{ R^{4} - 2 R^{2} x \cdot y + |x|^{2} |y|^{2} }{ R^{2} |x-y|^{2} }, \]
the last term equals
\begin{align*}
& \frac{1}{ |x-y|^{n-2s} } \frac{1}{ \zeta^{1-s} } \frac{1}{ (1+\zeta)^{n/2} } \nabla_{x} \zeta \\
& = \frac{1}{ |x-y|^{n-2s} } 
	\frac{ ( R^{2} - |x|^{2} )^{s-1} ( R^{2} - |y|^{2} )^{s-1} }{ R^{2(s-1)} |x-y|^{2(s-1)} }
	\frac{ R^{n} |x-y|^{n} }{  ( R^{4} - 2 R^{2} x \cdot y + |x|^{2} |y|^{2} )^{n/2} } \nabla_{x} \zeta \\
& = |x-y|^{2}  
	\frac{ ( R^{2} - |x|^{2} )^{s-1} ( R^{2} - |y|^{2} )^{s} R^{n} }
	{ R^{2s} ( R^{4} - 2 R^{2} x \cdot y + |x|^{2} |y|^{2} )^{n/2} } 
	\left( \frac{ - 2 x }{ | x - y |^{2} } - ( R^{2} - |x|^{2} ) \cdot 2 \frac{ x - y }{ | x - y |^{4} } \right).
\end{align*}
%
%
%
Hence,
\begin{equation}
- \nabla_{x} \frac{ \Phi_{n,s}(x,y) }{ |x-y|^{n-2s}  } 
= 2 x ( R^{2} - |x|^{2} )^{s-1} 
		\frac{ ( R^{2} - |y|^{2} )^{s} R^{n-2s} }{ ( R^{4} - 2 R^{2} x \cdot y + |x|^{2} |y|^{2} )^{n/2} } + 
		\frac{ x - y }{ | x - y |^{2} } P(y).
\label{eq:star}
\end{equation} 
with
\[ P(y) := ( n - 2 s ) \frac{ \Phi_{n,s}(x,y) }{ |x-y|^{n-2s}  }
		+ 2 \frac{ ( R^{2} - |x|^{2} )^{s} ( R^{2} - |y|^{2} )^{s} R^{n-2s} }
		{ ( R^{4} - 2 R^{2} x \cdot y + |x|^{2} |y|^{2} )^{n/2} }. \]

Due to the radial structure, the gradient in any point $x$ has the radial direction. 
The first term in (\ref{eq:star}) has this direction even for every $ y $.
In the second term, all the ``oscillations'' cancel by integration 
since the functions are symmetric with respect to the axis determined by $x$.

For $ x = 0 $, the integration yields 0.
If $ |x| \ge r_{0} $, the projection of $ x - y $ on $x$ is positive for every $ y \in B_{ r_{0} }(0) $.
Therefore, the second term in (\ref{eq:star}) yields a vector in the (positive) direction of $x$.

Finally, let us show that the same holds for $ x \in B_{ r_{0} }(0) $. 
Namely, let $ y_{x} := \frac{y \cdot x}{ |x| } $ stand for the signed length of the orthogonal projection of $y$ on $x$
and $ y_{x}^{\perp} := y - y_{x} \frac{x}{ |x| } $.
To every $ y \in B_{ r_{0} }(0) $ with $ y_{x} > |x| $ (whose contribution is in the opposite direction of $x$), 
there is $ \tilde{y} \in B_{ r_{0} }(0) $
with $ | \tilde{y} - x | = | y - x | $,  $ \tilde{y}_{x}^{\perp} = y_{x}^{\perp} $ 
and $ ( \tilde{y} - x )_{x} = - ( y - x )_{x} $.
Explictly, $ \tilde{y} = y - 2 ( y_{x} - |x| ) \frac{x}{ |x| } $.
Moreover, we denote by $ \tilde{ \zeta } $ the corresponding integration boundary.

For the norm it holds
\[ | \tilde{y} |^{2}
= | \tilde{y}_{x} |^{2} + | \tilde{y}_{x}^{\perp} |^{2} 
= | 2 |x| - y_{x} |^{2} + | y_{x}^{\perp} |^{2}
< | y_{x} |^{2} + | y_{x}^{\perp} |^{2}
= |y|^{2}. \]
Therefore, indeed $ \tilde{y} \in B_{ r_{0} }(0) $ and also $ \tilde{ \zeta } > \zeta $.
By rewriting
\begin{eqnarray*}
P(y)
&  =  & ( n - 2 s ) \frac{ \Phi_{n,s}(x,y) }{ |x-y|^{n-2s}  }
		+ 2 \frac{ ( R^{2} - |x|^{2} )^{s} ( R^{2} - |y|^{2} )^{s} R^{n-2s} }
		{ ( R^{4} - 2 R^{2} x \cdot y + |x|^{2} |y|^{2} )^{n/2} } \\
&  =  & \frac{ n - 2 s }{ | x - y |^{n-2s} } \int_{0}^{ \zeta } \frac{1}{ w^{1-s} ( 1 + w )^{n/2} } \ dw
		+ 2 \frac{ \zeta^{s} }{ | x - y |^{n-2s} ( 1 + \zeta )^{n/2} },
\end{eqnarray*}
we arrive at
\begin{eqnarray*}
P( \tilde{y} ) - P(y)
&  =  & \frac{1}{ | x - y |^{n-2s} } \left( 
		( n - 2 s ) \int_{ \zeta }^{ \tilde{ \zeta } } \frac{ w^{s-1} }{ ( 1 + w )^{n/2} } \ dw
		+ 2 \frac{ \tilde{ \zeta }^{s} }{ ( 1 + \tilde{ \zeta } )^{n/2} } - 2 \frac{ \zeta^{s} }{ ( 1 + \zeta )^{n/2} }
		\right) \\
&  =  & \frac{1}{ | x - y |^{n-2s} } \left( 
		( n - 2 s ) \int_{ \zeta }^{ \tilde{ \zeta } } \frac{ w^{s-1} }{ ( 1 + w )^{n/2} } \ dw
		+ 2 \int_{ \zeta }^{ \tilde{ \zeta } } 
		\frac{d}{dw} \frac{ w^{s} }{ (1+w)^{n/2} } \ dw
		\right) \\
&  =  & \frac{1}{ | x - y |^{n-2s} } \left( 
		( n - 2 s ) \int_{ \zeta }^{ \tilde{ \zeta } } \frac{ w^{s-1} }{ ( 1 + w )^{n/2} } \ dw
		+ 2 \int_{ \zeta }^{ \tilde{ \zeta } } 
		\left( \frac{ s w^{s-1} }{ (1+w)^{n/2} } - \frac{n}{2} \frac{ w^{s} }{ (1+w)^{n/2+1} } \right) \ dw
		\right) \\
&  =  & \frac{1}{ | x - y |^{n-2s} } \left( 
		n \int_{ \zeta }^{ \tilde{ \zeta } } \frac{ w^{s-1} }{ ( 1 + w )^{n/2} } \ dw
		- n \int_{ \zeta }^{ \tilde{ \zeta } } \frac{ w^{s} }{ (1+w)^{n/2+1} } \ dw
		\right) \\
&  =  & \frac{n}{ | x - y |^{n-2s} } 
		\int_{ \zeta }^{ \tilde{ \zeta } } \frac{ w^{s-1} }{ ( 1 + w )^{n/2+1} } \ dw \\
& \ge &	0.
\end{eqnarray*}
Hence,
\[ \int_{ B_{ r_{0} }(0) } \frac{ x - y }{ | x - y |^{2} } P(y) \y = \alpha x \]
for some $ \alpha > 0 $. 
Since the contribution of this term to the (radial) derivation is positive, we may neglect it for the lower bound.
Thus
\begin{eqnarray*}
| - \nabla b(x) |
&  =  & \left| \int_{ B_{ r_{0} }(0) } - \nabla_{x} G_{n,s}(x,y) \y \right| \\
& \ge & 2 |x| ( R^{2} - |x|^{2} )^{s-1} \frac{ \Gamma( \frac{n}{2} ) }{ 2^{2s} \pi^{1/n} \Gamma(s)^{2} }
		\int_{ B_{ r_{0} }(0) } 
		\frac{ ( R^{2} - |y|^{2} )^{s} R^{n-2s} }{ ( R^{4} - 2 R^{2} x \cdot y + |x|^{2} |y|^{2} )^{n/2} } \y \\
& \ge & 2 |x| ( R^{2} - |x|^{2} )^{s-1} \frac{ \Gamma( \frac{n}{2} ) }{ 2^{2s} \pi^{1/n} \Gamma(s)^{2} }
		| B_{ r_{0} }(0) | 
		\frac{ ( R^{2} - r_{0}^{2} )^{s} R^{n-2s} }
		{ ( R^{4} + 2 R^{2} R r_{0} + R^{2} r_{0}^{2} )^{n/2} } \\
& \ge & 2 s |x| ( R^{2} - |x|^{2} )^{s-1} \frac{ \Gamma( \frac{n}{2} ) }{ 2^{2s} \pi^{1/n} \Gamma(s)^{2} s }
		\frac{ \pi^{n/2} }{ \Gamma( \frac{n}{2} + 1 ) } \frac{ q^{n} ( 1 - q^{2} )^{s} }{ ( 1 + q )^{n} } \\
& \ge & \frac{2s}{n} \frac{ q^{n} ( 1 - q^{2} )^{s} }{ ( 1 + q )^{n} } | \nabla g(x) |,
\end{eqnarray*}
where we may the same assessment as above for the values of functions.
Therefore, if $ F_{1}, F_{2} $ and $q$ suffice the inequality above,
the solution $u$ has positive (radial) derivation. To summarize
\begin{prop}
\label{prop:local}
Let $ F_{1} , F_{2} > 0 $, $ n \ge 2 $ and $ q, s \in (0,1) $ fulfil 
\[ \frac{ F_{1} + F_{2} }{ F_{2} } \ge \frac{n}{2s} \frac{ ( 1 + q )^{n} }{ ( 1 - q^{2} )^{s} } \frac{1}{ q^{n} }. \]
Then, for any $ R > 0 $ and $ r_{0} := q R $, the solution to
\begin{eqnarray*} 
- ( - \Laplace )^{s} u(x) & = &  
\left\{ \begin{array}{cl} 
F_{1}, & \mbox{if } |x| < r_{0}, \\
- F_{2}, & \mbox{if } r_{0} < |x| < R, 
\end{array} \right. \\
u(x) & = & 0 \quad \mbox{if } |x| \ge R, 
\end{eqnarray*}
fulfils for all $ x \in B_{R}(0) $
\[ u(x) < 0 
\quad \mbox{and} \quad
\nabla u(x) = \alpha( |x| ) x 
\quad \mbox{with } \alpha \ge 0. \]
\end{prop}
We will also need an estimate of the minimal value of this solution $ u(0) $. Since
\[ u(0) 
= F_{2} \int_{ B_{R}(0) } G_{n,s}(0,y) \y 
- ( F_{1} + F_{2} ) \int_{ B_{ r_{0} }(0) } G_{n,s}(0,y) \y
\ge - F_{1} \int_{ B_{ r_{0} }(0) } G_{n,s}(0,y) \y, \]
we must find a good upper bound of $ b(0) = \int_{ B_{ r_{0} }(0) } G_{n,s}(0,y) \y $.
For $ n \ge 2 $ we can bound simply
\[ \int_{0}^{ \zeta } \frac{1}{ w^{1-s} } \frac{1}{ (1+w)^{n/2} } \ dw 
\le  \int_{0}^{1} \frac{1}{ w^{1-s} } \ dw 
+ \int_{1}^{ \i } \frac{1}{ w^{ n/2 + 1 - s } } \ dw 
= \frac{1}{s} + \frac{1}{ \frac{n}{2} - s } 
= \frac{1}{ s ( \frac{n}{2} - s ) }.\]
Then
\begin{eqnarray*}
b(0) 
&  =  & \int_{ B_{ r_{0} }(0) } G_{n,s}(0,y) \y \\
& \le & \frac{ \Gamma( \frac{n}{2} ) }{ 2^{2s} \pi^{1/n} \Gamma(s)^{2} } \frac{1}{ s ( \frac{n}{2} - s ) }
		\int_{ B_{ r_{0} }(0) } \frac{1}{ |y|^{ n - 2s } } \y \\
&  =  & \frac{ \Gamma( \frac{n}{2} ) }{ 2^{2s} \pi^{1/n} \Gamma(s)^{2} } \frac{1}{ s ( \frac{n}{2} - s ) }
		\int_{0}^{ r_{0} }  \frac{1}{ \rho^{ n - 2s } } {\cal H}^{n-1}( \partial B_{ \rho } ) \ d \rho \\
&  =  & \frac{ \Gamma( \frac{n}{2} ) }{ 2^{2s} \pi^{1/n} \Gamma(s)^{2} } \frac{1}{ s ( \frac{n}{2} - s ) }
		\frac{ 2 \pi^{n/2} }{ \Gamma( \frac{n}{2} ) }
		\int_{0}^{ r_{0} }  \frac{1}{ \rho^{ 1 - 2s } } \ d \rho \\	
&  =  & \frac{ 2 \pi^{n/2} }{ 2^{2s} \pi^{1/n} \Gamma(s)^{2} s ( \frac{n}{2} - s ) }
		\frac{ r_{0}^{2s} }{2s}.
\end{eqnarray*}
Therefore, 
\[ \min u = u(0) 
\ge - \frac{ \pi^{n/2} }{ 2^{2s} \pi^{1/n} \Gamma(s)^{2} s^{2} ( \frac{n}{2} - s ) } F_{1} r_{0}^{2s} . \]
%
\section{Decomposition} \label{sect:decomposition}
Positions of obstacles are random. 
In order to construct a supersolution, however, we must have some control on them.
Therefore, in this section we introduce a suitable decomposition of the space.
We will later see that by taking the right scaling
we will find in sufficiently many members of this decomposition obstacles that are strong enough.

Let us first decompose the base space $ \R^{n} $.
Let for each $ a = ( a_{1} , \ldots , a_{n} ) \in \Z^{n} $ be
\[ Q_{a} := \prod_{ i = 1 }^{n} 
\big[ a_{i} ( l + d ) - \tfrac{l}{2} + r_{1} , a_{i} ( l + d ) + \tfrac{l}{2} - r_{1} \big] \]
and
\[ \hat{Q}_{a} := \prod_{ i = 1 }^{n} \big[ a_{i} ( l + d ) - \tfrac{l}{2} , a_{i} ( l + d ) + \tfrac{l}{2} \big] \]
for some (still arbitrary) $ l > 2 r_{1} $ and $ d > 0 $, as depicted in Figure~\ref{fig:base}.
\begin{figure}
\begin{center}
\includegraphics[width=.5\textwidth]{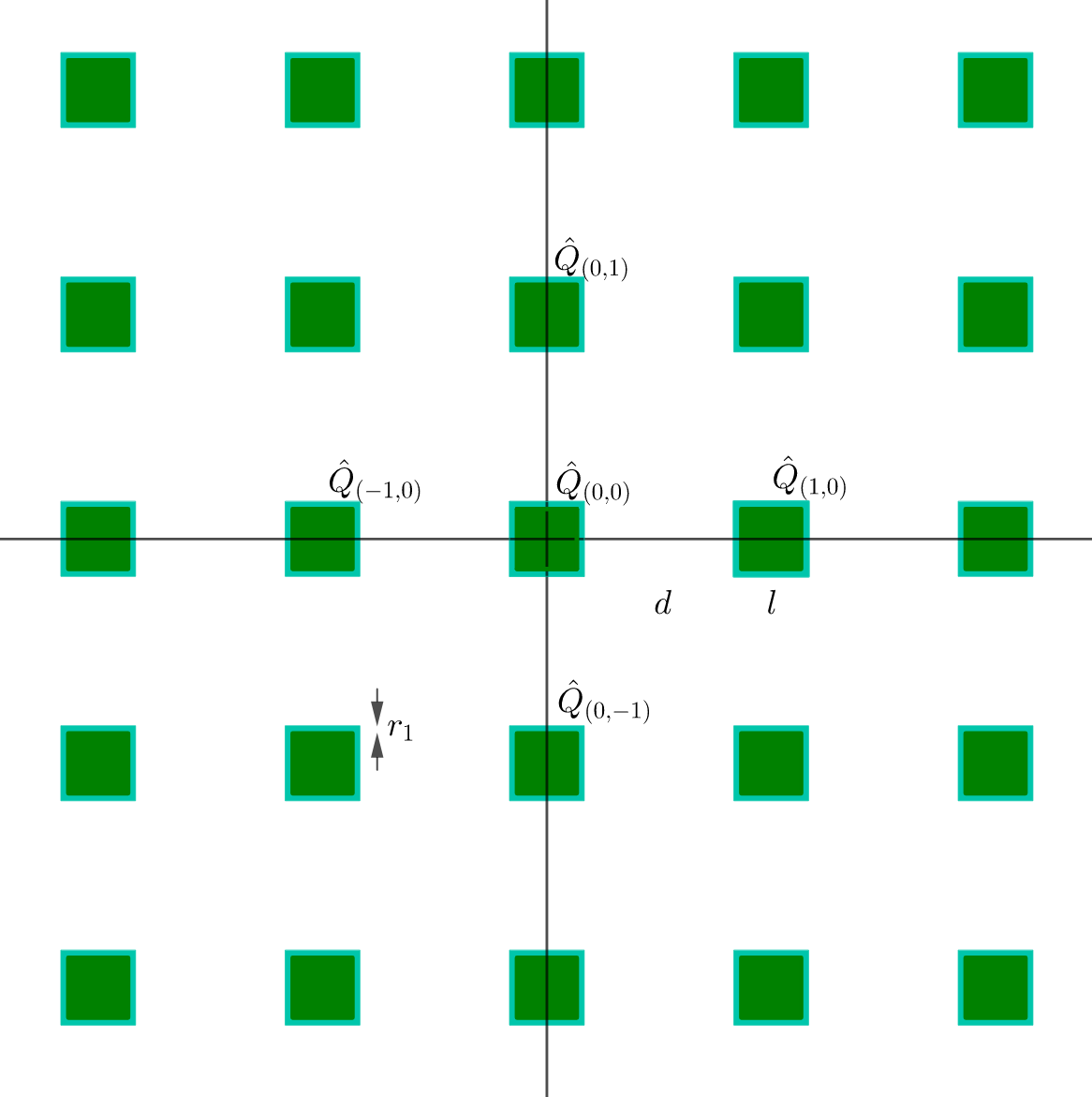}
\end{center}
\caption{Decomposition of the base space for $ n = 2 $} \label{fig:base}
\end{figure}

Let us choose arbitrary $ R \ge \tfrac{d}{2} + l - r_{1} $ and $ x_{a} \in Q_{a} $ for every $ a \in \Z^{n} $.
If $ F_{1} , F_{2} $ and $ 0 < s < 1 $ suffice the condition in Proposition~\ref{prop:local}
with $ u_{\rm local} $ being the corresponding solution, we define
\[ u_{\rm flat} : \R^{n} \to \R,
\quad
u_{\rm flat}(x) := \min_{ a \in \Z^{n} } u_{\rm local}( x - x_{a} ). \]
The points $ \{ x_{a} : a \in \Z^{n} \} $ induce a Voronoi diagram.
Since the function $ u_{\rm local} $ is radial, negative on $ B_{R}(0) $ and grows away from the origin, 
for every point $ x \in \R^{n} $ that lies in the interior of the Voronoi cell belonging to $ x_{a} $,
it holds
\[ u_{\rm flat}(x) = u_{\rm local}( x - x_{a} ) \]
and 
\begin{eqnarray*} 
- ( - \Laplace )^{s} u_{\rm flat}(x) 
&  =  & P.V. \int_{ \R^{n} } \frac{ u_{\rm flat}(x+z) - u_{\rm flat}(x) }{ |z|^{ n + 2 s } } \\
& \le & P.V. \int_{ \R^{n} } \frac{ u_{\rm local}( x + z - x_{a} ) - u_{\rm local}( x - x_{a} ) }{ |z|^{ n + 2 s } } \\
&  =  & - ( - \Laplace )^{s} u_{\rm local}( x - x_{a} ) \\
& \le &
\left\{ \begin{array}{cl} 
F_{1}, & \mbox{if } | x - x_{a} | < r_{0}, \\
- F_{2}, & \mbox{if } r_{0} < | x - x_{a} | < R. 
\end{array} \right. 
\end{eqnarray*}
On the boundary of Voronoi cells, the condition for viscosity solution is trivially fulfilled since
there is no $ C^{2} $-function that lies locally below $ u_{\rm flat} $ and touches its graph in these points.

We bear in mind that here we assumed the positions of obstacles to be $ ( x_{a} , u_{\rm local}(0) ) $.
We still must make sure to really find sufficiently many of them and to lift this function to their actual height
since now they are lying below $ \R^{n} \times \{0\} $.

Let us for fixed but still arbitrary $ h > 0 $ define cuboids
\[ Q_{a,j} := Q_{a} \times [ ( j - 1 ) h + r_{1} , j h + r_{1} ]
. \]
We chose these cuboids so that, if an obstacle lies in some $ Q_{a,j} $, 
then its entire force acts within $ \hat{Q}_{a} \times (0,\i) $.

Suppose at the beginning we have a slightly perturbed horizontal hyperplane. 
More precisely, let there be $ U \in C^{2}( \R^{n} ) $ such that
\[ U(x) = \nu \cdot x + r(x) \]
with 
\[ \nu \in \R^{n}, 
\quad 
\sup_{ x \in \R^{n} } | \nabla U(x) | =: \| \nabla U \|_{\i} < 1 
\quad \mbox{and} \quad 
\sup_{ x \in \R^{n} } | ( - \Laplace )^{s} r(x) | =: \| ( - \Laplace )^{s} U \|_{\i} < 1. \] 
We define a bijection
\[ {\cal U} : \R^{n} \times [ 0 , \i ) \to \{ (x,y) \in \R^{n} \times \R : y \ge U(x) \},
\quad 
{\cal U}(x,y) := ( x , y + U(x) ), \]
with the obvious inverse $ {\cal U}^{-1}(x,y) = ( x , y - U(x) ) $.
Thus  
\[ {\mathcal Q}_{a,j} := {\cal U} ( Q_{a,j} ), \quad a \in \Z^{n}, j \in \N \]
determine a decomposition of the half-space above the surface $ \{ (x,y) : y = U(x) \} $.
The two decompositions look as depicted in Figure~\ref{fig:f}.
\begin{figure}
\begin{center}
\includegraphics[width=.48\textwidth]{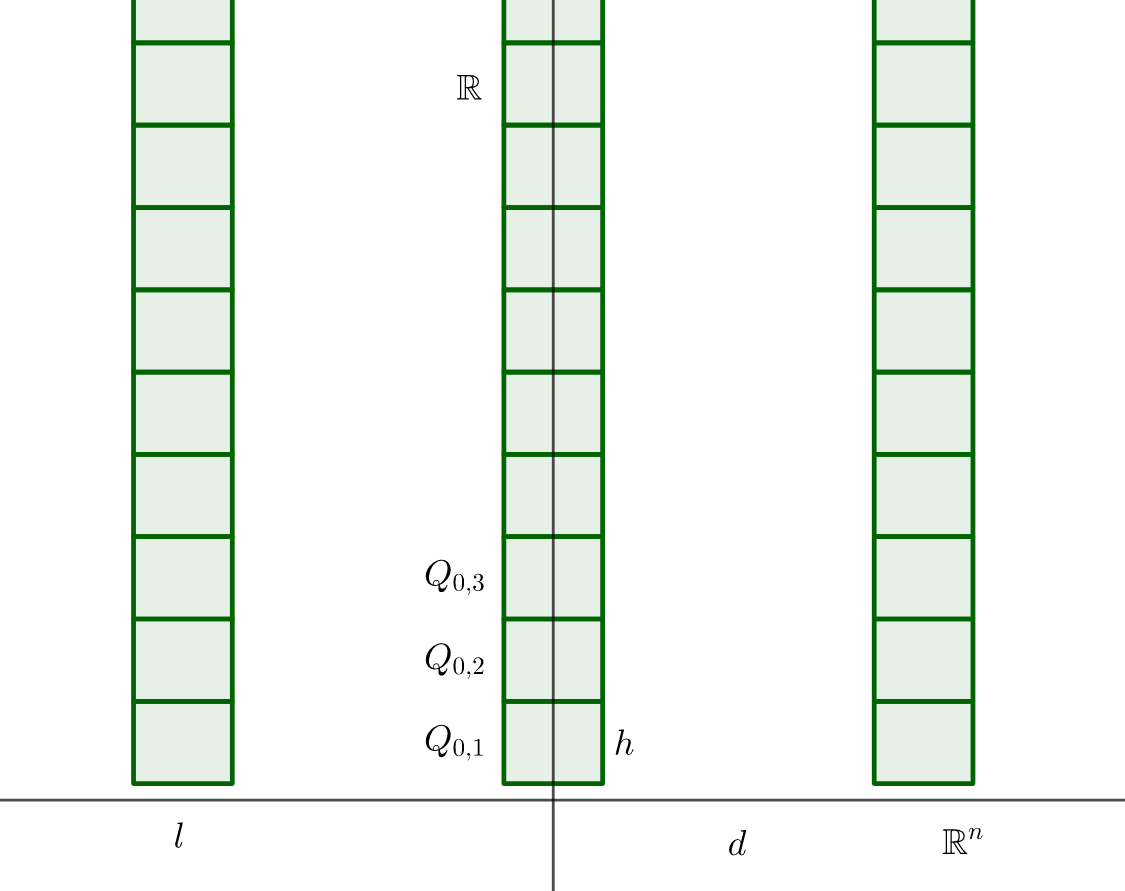}
\quad
\includegraphics[width=.48\textwidth]{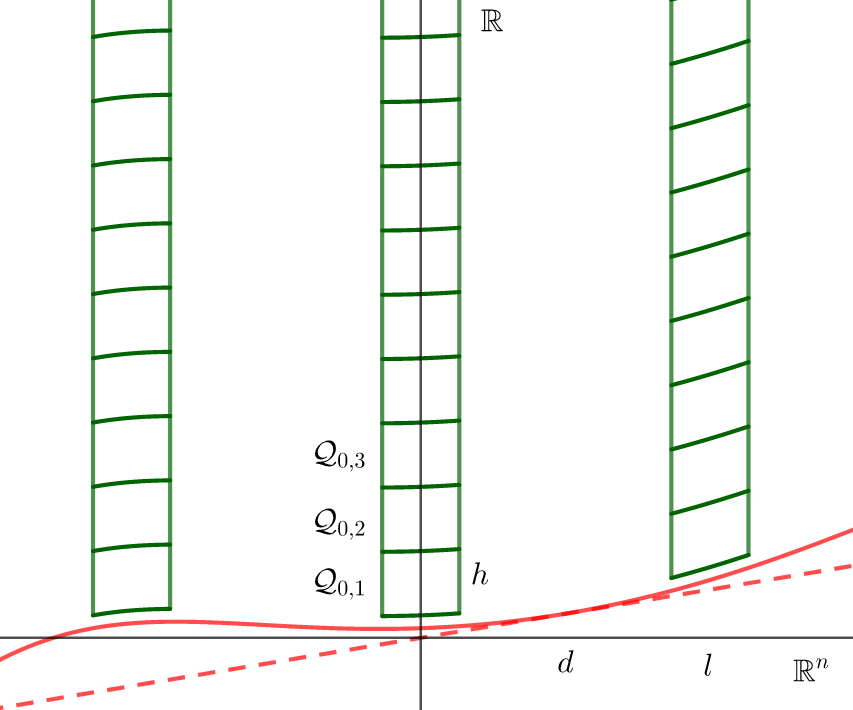}
\end{center}
\caption{Cross-sections of the decomposition of $ \R^{n+1} $ above the horizontal hyperplane (left) 
and slightly perturbed one (right) \label{fig:f}}
\end{figure}

Obviously, $ {\mathcal Q}_{a,j} $ and $ Q_{a,j} $ have volume $ ( l - 2 r_{1} )^{n} h $
and their orthogonal projection to $ \R^{n} \times \{ 0 \} $ is $ Q_{a} $.

Our idea, in order to simplify the contruction, is to regard the flat case and add the function $ U $ at the end.
However, by this ``flattening'' with $ {\cal U}^{-1} $, the obstacles get deformed. 
We suppose that the obstacles have full strength in $ (n+1) $-dimensional cubes with side $ 2 r_{0} $.
Let us take such an obstacle with center $ ( x_{0} , y_{0} ) $.
What is the height $ \eta $ of the cylinder with radius $ r_{0} $ 
centered at $ {\cal U}^{-1}( x_{0} , y_{0} ) $ so that its image under $ {\cal U} $ will lie in such an obstacle?
More precisely, when is 
\begin{multline*}
{\cal U} \Big( \{ ( x_{0} + \xi , y_{0} - U( x_{0} ) + \eta ) \in \R^{n} \times \R : 
| \xi | \le r_{0}, \ | \eta | \le \eta_{0} \} \Big) \subset \\
\subset \{ ( x , y ) \in \R^{n} \times \R : \| ( x , y ) - ( x_{0} , y_{0} ) \|_{\i} \le r_{0} \}? 
\end{multline*}
%
%
Since only the last component changes, for arbitrary $ | \xi | \le r_{0} $ and $ | \eta | \le \eta_{0} $, this reads as
\[ r_{0} 
\ge | y_{0} - U( x_{0} ) + \eta + U( x_{0} + \xi ) - y_{0} |
= | \eta + U( x_{0} + \xi ) - U( x_{0} ) |. \]
We have
\[ | U( x_{0} + \xi ) - U( x_{0} ) | 
\le \sup_{ x \in \R^{n} } | \nabla U(x) | \cdot | \xi |
\le \| \nabla U \|_{\i} r_{0}. \]
Thus this will surely be true if $ \eta \le ( 1 -  \| \nabla U \|_{\i} ) r_{0} $ \label{page:obstacle-height} since then
\[ | \eta + U( x_{0} + \xi ) - U( x_{0} ) |
\le | U( x_{0} + \xi ) - U( x_{0} ) | + | \eta | \le r_{0}. \]
%

\section{Percolation} \label{sect:percolation}
Now we adress the problem of finding obstacles with appropriate positions and sufficient strengths.
We choose any $ S > 0 $ with $ \mathbf{P}( f_{0} \ge S ) =: \mu_{S} > 0 $. 
Our goal is to get an array of obstacles such that
\begin{itemize}
\item 
for each $ a \in \Z^{n} $ there is an obstacle inside $ \hat{Q}_{a} \times \R $ above the graph of $U$
in order for local solutions in the definition of $ u_{\rm flat} $ to intersect in their negative regions,
\item 
their heights locally differ mildly so that, by lifting the function $ u_{\rm flat} $ to their positions,
we still can control its fractional Laplacian,
\item
their strength is at least $S$.
\end{itemize} 
The main tool to get such an array will be Theorem 2.1 from \cite{DondlScheutzowThrom}:
\begin{theo}
\label{theo:percolation}
Suppose $ z \in \Z^{n+1} $ is open with probability $ p \in (0,1) $ and closed otherwise, 
with different sites receiving independent states. 
The corresponding probability measure on the sample space $ \Omega = \{0, 1\}^{ \Z^{n+1} } $
is denoted by $ \mathbf{P}_{p} $.
For every nondecreasing function $ H : \N \to \N $ with
\[ \liminf_{ k \to \i } \frac{ H(k) }{ \log k } > 0, \] 
there exists $ p_{H} = p_{H}(n) \in (0,1) $ such that for every $ p \in ( p_{H} , 1 ) $
there exists a.s.~a (random) function $ y : \Z^{n} \to \N $ with the following properties:
\begin{itemize}
\item
For each $ a \in \Z^{n} $, the site $ ( a, y_{a} ) \in \Z^{n+1} $ is open.
\item
For any $ a , b \in \Z^{n} $, $ a \ne b $, it holds $ | y_{a} - y_{b} | \le H( \| a - b \|_{1} ) $.
\end{itemize}
Moreover, if we choose some $ p \in ( 1 , p_{H} ) $ 
and take the smallest function $ y $ with the above properties, there exists a constant $ C_{n,p,H} $
such that for all $ m \in \N \cup \{0\} $ 
\[ \mathbf{P}( y_{0} > m ) \le C_{n,p,H} \frac{ 2^{m} (1-p)^{m} }{ 2p-1 }. \] 
\end{theo}
In our case, the sides $ (a,j) \in \Z^{n} \times \N $ correspond to the cuboids $ {\mathcal Q}_{a,j} $.
We declare a side to be open if $ {\mathcal Q}_{a,j} $ contains the centre of an obstacle with strength at least $S$.
Having a Poisson point process, we know that the percolation result is applicable if
\[ 1 - \exp( - \lambda | {\mathcal Q}_{a,j} | \mu_{S} ) = 1 - \exp( - \lambda h ( l - 2 r_{1} )^{n} \mu_{S} ) > p_{H}. \]
Thus we arrive at another condition on the scaling, in this case on $l$ and $h$.
For $H$, we may take $ H(k) = \lfloor k^{ \alpha } \rfloor $ for any $ \alpha \in (0,1] $.  

\section{Lifting function} \label{sect:lifting}
As already announced, we also have to construct a suitable lifting function.
For that purpose, we adapt Proposition 3.22 in \cite{DondlScheutzowThrom} to the higher dimensional case.
We stress that the statements and the idea of the proof are mutatis mutandis the same.
\begin{prop}
\label{prop:lifting-function}
Let $ h, d, l > 0 $ and $ s \in (0,1) $. 
For $ \Lambda : \Z^{n} \to \R $ such that
\[ | \Lambda(a) - \Lambda(b) | \le 2 h \| a - b \|_{1}^{ \alpha } \]
with $ 0 < \alpha < 2 s $, 
there exist a smooth function $ u_{\rm lift} : \R^{n} \to \R $ 
and constants $ C_{0}, C_{1}, C_{2} $ depending only on $ n, s , \alpha $ such that:
\begin{itemize}
\item
$ u_{ \rm lift }(x) = \Lambda(a) $ if $ x \in \hat{Q}_{a} $ for some $ a \in \Z^{n} $,
\item
$ \| D^{2} u_{ \rm lift } \|_{ L^{\i} } \le C_{0} \frac{h}{ d^{2} } $,
\item
$ | ( - \Laplace )^{s} u_{ \rm lift }(x) |
\le C_{1} ( d + l )^{ 2 - 2 s } \frac{h}{ d^{2} } 
+ C_{2} \frac{h}{ ( d + l )^{ 2 s } } $.
\end{itemize}
\end{prop}
\begin{proof}
We define $ u_{\rm lift} : \R^{n} \to \R $ 
as $ \frac{d}{2} $-mollification of the piecewise constant rescaled extension of $ \Lambda $ namely
\[ \tilde{ \Lambda }(x) 
:= \Lambda \left( \left\lfloor \frac{ x_{1} }{ l+d } + \frac{1}{2} \right\rfloor , \ldots , 
\left\lfloor \frac{ x_{n} }{ l+d } + \frac{1}{2} \right\rfloor \right) \]
Hence, $ u_{\rm lift} := \eta_{d/2} * \tilde{ \Lambda } $. 
The first two properties then follow from the standard mollification results.
As for the third, consider without loss of generality 
that $ x $ belongs to the cube with center at the origin and with side $ d + l $, 
hence $ x \in \tilde{Q} := [ - \frac{ d + l }{2} , \frac{ d + l }{2} ]^{n} = \frac{d+l}{l} \hat{Q}_{0} $. 
According to Lemma 3.2 from \cite{DPV} and Proposition A.12 from \cite{Throm}, 
for smooth function with sublinear growth, we may write
\[ ( - \Laplace )^{s} u_{ \rm lift }(x) 
= - \frac{C}{2} \int_{ \R^{n} } 
\frac{ u_{ \rm lift }(x+y) + u_{ \rm lift }(x-y) - 2 u_{ \rm lift }(x) } { |y|^{n+2s} } \y. \]
We split the integration domain into $ 3 \tilde{Q} $ and $ \R^{n} \setminus 3 \tilde{Q} $.
\item
For all $ y \not\in 3 \tilde{Q} $, it holds
\begin{eqnarray*}
| u_{ \rm lift }(x+y) - u_{ \rm lift }(x) | 
& \le & | u_{ \rm lift }(x+y) - \tilde{ \Lambda }(x+y) | 
		+ | \tilde{ \Lambda }(x+y) - \tilde{ \Lambda }(x) | 
		+ | \tilde{ \Lambda }(x) - u_{ \rm lift }(x) | \\
& \le & ( n + 1 ) h + 2 h \left( n + \frac{ | y | }{ l + d } \right)^{ \alpha } + ( n + 1 ) h \\
& \le & 4 ( n + 1 ) h \left( \frac{ | y | }{ l + d } \right)^{ \alpha }.
\end{eqnarray*}
Thus,
\begin{eqnarray*}
\int_{ \R^{n} \setminus 3 \tilde{Q} } 
\frac{ | u_{ \rm lift }(x+y) + u_{ \rm lift }(x-y) - 2 u_{ \rm lift }(x) | } { |y|^{n+2s} } \y
& \le & \int_{ \R^{n} \setminus 3 \tilde{Q} } 
		2 \frac{ 4 ( n + 1 ) h\left( \frac{ | y | }{ l + d } \right)^{ \alpha } } { |y|^{n+2s} } \y \\
&  =  & \int_{ \R^{n} \setminus 3 \tilde{Q} } 
		8 ( n + 1 ) h \frac{ | y |^{ \alpha - n - 2s } } { ( l + d )^{ \alpha } } \y \\
& \le & \int_{ \R^{n} \setminus B_{ 3( l + d )/2 }  } 
		8 ( n + 1 ) h \frac{ | y |^{ \alpha - n - 2s } } { ( l + d )^{ \alpha } } \y,
\end{eqnarray*}
and
\begin{eqnarray*}
\int_{ \R^{n} \setminus B_{ 3( l + d )/2 }  } | y |^{ \alpha - n - 2s } \y 
&  =  & \int_{ 3( l + d )/2  }^{\i}
		\rho^{ \alpha - n - 2s } {\cal H}^{n-1}( \partial B_{ \rho } ) \ d \rho  \\
&  =  & \frac{ 2 \pi^{n/2} }{ \Gamma( \frac{n}{2} ) }
		\int_{ 3( l + d )/2  }^{\i} \rho^{ \alpha - 1 - 2s } \ d \rho  \\
&  =  & \frac{ 2 \pi^{n/2} }{ \Gamma( \frac{n}{2} ) }
		\frac{ ( \frac{3}{2}( l + d ) )^{ \alpha - 2 s } }{ 2 s - \alpha }.
\end{eqnarray*}
For the remaining part, we assess
\begin{eqnarray*}
\int_{ 3 \tilde{Q} } \frac{ | u_{ \rm lift }(x+y) + u_{ \rm lift }(x-y) - 2 u_{ \rm lift }(x) | } { |y|^{n+2s} } \y
& \le & \int_{ 3 \tilde{Q} } 
		2 \frac{ \| D^{2} u_{ \rm lift } \|_{ L^{\i} } |y|^{2} } { |y|^{n+2s} } \y \\
& \le & \int_{ B_{ 3 \sqrt{n} ( l + d )/2 } } 
		2 \| D^{2} u_{ \rm lift } \|_{ L^{\i} } |y|^{2-n-2s} \y,
\end{eqnarray*}
and
\begin{eqnarray*}
\int_{ B_{ 3 \sqrt{n} ( l + d )/2 } } |y|^{2-n-2s} \y 
&  =  & \int_{0}^{ 3 \sqrt{n} ( l + d )/2 } \rho^{2-n-2s} {\cal H}^{n-1}( \partial B_{ \rho } ) \ d \rho \\ 
&  =  & \frac{ 2 \pi^{n/2} }{ \Gamma( \frac{n}{2} ) } \int_{0}^{ 3 \sqrt{n} ( l + d )/2 } \rho^{1-2s} \ d \rho \\ 
&  =  & \frac{ 2 \pi^{n/2} }{ \Gamma( \frac{n}{2} ) } \frac{ ( \frac{ 3 \sqrt{n} }{2}( l + d ) )^{2-2s} }{ 2 - 2s }.
\end{eqnarray*}
Hence,
\begin{eqnarray*}
| ( - \Laplace )^{s} u_{ \rm lift }(x) |
& \le & \frac{C}{2} 
	 	\left( 2 \| D^{2} u_{ \rm lift } \|_{ L^{\i} } \frac{ 2 \pi^{n/2} }{ \Gamma( \frac{n}{2} ) } 
				\frac{ ( \frac{ 3 \sqrt{n} }{2}( l + d ) )^{2-2s} }{ 2 - 2s }
				+ \frac{ 8 ( n + 1 ) h } { ( l + d )^{ \alpha } } \frac{ 2 \pi^{n/2} }{ \Gamma( \frac{n}{2} ) }
				\frac{ ( \frac{3}{2}( l + d ) )^{ \alpha - 2 s } }{ 2 s - \alpha } \right) \\
& \le & C \left( C_{0} \frac{h}{ d^{2} } \frac{ 2 \pi^{n/2} }{ \Gamma( \frac{n}{2} ) } 
				\frac{ ( \frac{ 3 \sqrt{n} }{2} )^{2-2s} }{ 2 - 2s } ( l + d )^{2-2s}
				+ \frac{ 8 ( n + 1 ) \pi^{n/2} }{ \Gamma( \frac{n}{2} ) } \frac{ ( \frac{3}{2} )^{ \alpha - 2 s } }{ 2 s - \alpha }
	 			\frac{ h } { ( l + d )^{ 2 s } } \right).
\end{eqnarray*}
Therefore, we may choose
\[ C_{1} := C C_{0} \frac{ 2 \pi^{n/2} }{ \Gamma( \frac{n}{2} ) } \frac{ ( \frac{ 3 \sqrt{n} }{2} )^{2-2s} }{ 2 - 2s },
\quad
C_{2} := C \frac{ 8 ( n + 1 ) \pi^{n/2} }{ \Gamma( \frac{n}{2} ) } 
		\frac{ ( \frac{3}{2} )^{ \alpha - 2 s } }{ 2 s - \alpha }. \]
\end{proof}
\section{Scaling} \label{sect:scaling}
\begin{trivlist}
\item
Our supersolution will be
\[ u(x) := u_{\rm flat}(x) + u_{\rm lift}(x) + U(x). \] 
However, we still have many parameters to choose, namely $ q , F_{1} , F_{2} , S , l , d , h , \alpha $ and function $H$,
in such a way that they suffice the conditions derived in previous sections. 
We will suppose $ \| \nabla U \|_{\i} < 1 $ and derive a condition for $ \| ( - \Laplace )^{s} U \|_{\i} $.
\item
First we fix
\begin{itemize}
\item any $ S > 0 $ such that $ \mathbf{P}( f_{0} \ge S ) = \mu_{S} > 0 $,
\item an $ \alpha \in (0,1] $ with $ \alpha < 2s $ (e.g.~$ \alpha = s $),
\item the function $ H(k) := \lfloor k^{\alpha} \rfloor $.
\end{itemize}
Let $ p_{\alpha} $ be the probability $ p_{H} $ for the function $H$.
As we have already seen, we must take such $ h , l > 0 $ that
\[ 1 - \exp( - \lambda h ( l - 2 r_{1} )^{n} \mu_{S} ) > p_{ \alpha } \]
or 
\begin{equation}
\boxed{ h ( l - 2 r_{1} )^{n} > - \frac{1}{ \lambda \mu_{S} } \log ( 1 - p_{ \alpha } ) } 
\label{eq:1}
\end{equation}  
\begin{itemize}
\item We choose  $ d := l := \frac{r_{0}}{ 2 q \sqrt{n} }  $, and thus $ R = 2 l \sqrt{n} = \frac{1}{q} r_{0} $.
\end{itemize}
We must make sure that the solution does not ``fall out'' of an obstacle.
As we derived on the page~\pageref{page:obstacle-height},
it must therefore hold $ | u_{ \rm local } | < r_{0} ( 1 - \| \nabla U \|_{\i} ) $. 
According to the estimate after Proposition~\ref{prop:local}, this will hold if 
\begin{equation} 
\boxed{ \frac{ \pi^{n/2} }{ 2^{2s} \pi^{1/n} \Gamma(s)^{2} s^{2} ( \frac{n}{2} - s ) } F_{1} r_{0}^{2s} 
\le r_{0} ( 1 - \| \nabla U \|_{\i} ) }
\label{eq:2}
\end{equation} 
The local solutions are according to Proposition~\ref{prop:local} non-positive and radially increasing, 
which is needed so that their minimum is still a supersolution,
if
\begin{equation}
\boxed{ \frac{ F_{1} + F_{2} }{ F_{2} } 
\ge \frac{n}{2s} \frac{ ( 1 + q )^{n} }{ ( 1 - q^{2} )^{s} } \frac{1}{ q^{n} } }
\label{eq:3}
\end{equation} 
Lastly, we have to lift the flat supersolution to the obstacles. 
Suppose we spend $ F^{*} := \frac{1}{2} \min \{ S - F_{1} , F_{2} \} $ (with $ F_{1} , F_{2} $ yet to be chosen) on it.
Then we must achieve
\[ | ( - \Laplace )^{s} u_{ \rm lift } + ( - \Laplace )^{s} U |
\le F^{*}.\]
According to Proposition~\ref{prop:lifting-function}, this will surely hold if
\begin{equation}
\boxed{ C_{1} ( d + l )^{ 2 - 2 s } \frac{h}{ d^{2} } 
+ C_{2} \frac{h}{ ( d + l )^{ 2 s } } + \| ( - \Laplace )^{s} U \|_{\i}
\le \frac{1}{2} \min \{ S - F_{1} , F_{2} \} }
\label{eq:4} 
\end{equation} 
We must determine whether we may simultaneously fulfil the inequalities~(\ref{eq:1})-(\ref{eq:4}). 
Let us therefore simplify them.
Since we will choose $ l > 4 r_{1} $ (and therefore $ l - 2 r_{1} > \frac{l}{2} $), the first will be fulfilled if 
\[ h l^{n} \ge - \frac{ 2^{n}  }{ \lambda \mu_{S} } \log ( 1 - p_{ \alpha } ). \]
Employing $ 2 l \sqrt{n} = \frac{ r_{0} }{q} $, we arrive at
\begin{equation}
\frac{ h }{ q^{n} } > \frac{ A_{1} }{ \lambda \mu_{S} }  
\quad \mbox{with} \quad
A_{1} := - \frac{ 2^{ 2n } \sqrt{n}^{n} }{ r_{0}^{n} } \log ( 1 - p_{\alpha} ).
\label{eq:1-1} 
\end{equation} 
For the second, we must get
\begin{equation}
F_{1} \le A_{2} ( 1 - \| \nabla U \|_{\i} ) \quad \mbox{with} \quad
A_{2} 
:= r_{0}^{1-2s} 
\frac{ 2^{2s} \pi^{1/n} \Gamma(s)^{2} s^{2} ( \frac{n}{2} - s ) }{ \pi^{n/2} }. 
\label{eq:2-1} 
\end{equation}
Since $ q = \frac{ r_{0} }{ 2 l \sqrt{n} } \le \frac{ r_{0} }{ 8 r_{1} \sqrt{n} } \le \frac{1}{8n} $,
the third will surely be fulfilled if
\begin{equation}
\frac{ F_{1} }{ F_{2} } q^{n}
\ge A_{3} := \frac{n}{2s} \left( \frac{9}{8} \right)^{n} \frac{64}{63}. 
\label{eq:3-1} 
\end{equation}
As for the last, it should hold
\[ ( 4 C_{1} + C_{2} ) \frac{h}{ ( 2 l )^{2s} } + \| ( - \Laplace )^{s} U \|_{\i}  
\le \frac{1}{2} \min \{ S - F_{1} , F_{2} \}, \]
or, with $ A_{4} := \frac{ 4 C_{1} + C_{2} }{ r_{0}^{2s} } n^{s} $, 
\begin{equation}
A_{4} h q^{2s} + \| ( - \Laplace )^{s} U \|_{\i} 
\le \frac{1}{2} \min \{ S - F_{1} , F_{2} \}. 
\label{eq:4-1} 
\end{equation}
The question is if there exist such $ q \in (0,1) $ and $ F_{1} , F_{2} , h > 0 $.
Let $q$ be free and set 
\begin{itemize}
\item
$ F_{1} := \min \{ A_{2} ( 1 - \| \nabla U \|_{\i} ) , \tfrac{S}{2} \}, $
\item
$ F_{2} := \tfrac{ \min \{ A_{2} ( 1 - \| \nabla U \|_{\i} ) , \tfrac{S}{2} \} }{ A_{3} } q^{n}. $
\end{itemize}
Thus, the inequalities~(\ref{eq:2-1})~and~(\ref{eq:3-1}) are fulfilled.
Obviously, since $ q < \frac{1}{2} $
\[ \min \{ S - F_{1} , F_{2} \} = F_{2} 
= \tfrac{ \min \{ A_{2} ( 1 - \| \nabla U \|_{\i} ) , \tfrac{S}{2} \} }{ A_{3} } q^{n}. \]
If we suppose that in (\ref{eq:4-1}) every summand contributes up to one half of the upper bound,
then the inequalities~(\ref{eq:1-1})~and~(\ref{eq:4-1}) will be fulfilled if 
\[ \frac{ h }{ q^{n} } \ge \frac{ A_{1} }{ \lambda \mu_{S} },
\quad
\frac{ h }{ q^{n} } q^{2s} 
\le \frac{ \min \{ A_{2} ( 1 - \| \nabla U \|_{\i} ) , \tfrac{S}{2} \} }{ 4 A_{3} A_{4} },
\quad
\| ( - \Laplace )^{s} U \|_{\i} 
\le \frac{ \min \{ A_{2} ( 1 - \| \nabla U \|_{\i} ) , \tfrac{S}{2} \} }{ 4 A_{3} } q^{n}. \]
Therefore, we set 
\begin{itemize}
\item 
$ {\ds h := \frac{ A_{1} }{ \lambda \mu_{S} } q^{n}. } $
\end{itemize}
It must then hold
\[ \frac{ A_{1} }{ \lambda \mu_{S} } q^{2s}
\le \frac{ \min \{ A_{2} ( 1 - \| \nabla U \|_{\i} ) , \tfrac{S}{2} \} }{ 4 A_{3} A_{4} }. \]
It is always possible to achieve this by setting
\begin{itemize}
\item 
$ {\ds q := \left(
\frac{ \min \{ A_{2} ( 1 - \| \nabla U \|_{\i} ) , \tfrac{S}{2} \} }{ 4 A_{1} A_{3} A_{4} } \lambda \mu_{S}
\right)^{1/2s}. } $
\end{itemize}
(We must pay attention also to $ q < \frac{ r_{0} }{ 8 r_{1} \sqrt{n} } $.)
Now, we were able to choose all the parameters such that they suffice the inequalities~(\ref{eq:1})-(\ref{eq:4})
if we additionally suppose
\[ \| ( - \Laplace )^{s} U \|_{\i} 
\le \left( 
\frac{ \min \{ A_{2} ( 1 - \| \nabla U \|_{\i} ) , \tfrac{S}{2} \} }{ 4 A_{3} } 
\right)^{ 1 + \frac{n}{2s} }
\left( \frac{ \lambda \mu_{S} }{ A_{1} A_{4} } \right)^{ \frac{n}{2s} }. \]
Hence, there exists a constant $ A_{0} $ depending on $ n , s , \alpha, r_{0} $ and $ r_{1} $ such that the condition reads
\[ \| ( - \Laplace )^{s} U \|_{\i} 
\le A_{0}  \left( \min \{ A_{2} ( 1 - \| \nabla U \|_{\i} ) , \tfrac{S}{2} \} \right)^{ 1 + \frac{n}{2s} }
\left( \lambda \mu_{S} \right)^{ \frac{n}{2s} }. \]
Let us write a stronger though simpler condition.
First, we may demand instead
\[ \| ( - \Laplace )^{s} U \|_{\i} 
\le A_{0}  \left( \min \{ A_{2} , \tfrac{S}{2} \} \right)^{ 1 + \frac{n}{2s} }
\left( \lambda \mu_{S} \right)^{ \frac{n}{2s} } ( 1 - \| \nabla U \|_{\i} )^{ 1 + \frac{n}{2s} } . \]
We choose $S$ that maximizes the expression on the right (and $ \alpha = s $). 
Thus, there exists a constant $ C = C( n , s , r_{0} , r_{1} , \lambda , f_{0} ) $ such that if
\[ \| ( - \Laplace )^{s} U \|_{\i} 
\le C ( 1 - \| \nabla U \|_{\i} )^{ 1 + \frac{n}{2s} }, \]
then the pinning occurs.
\begin{theo}
\label{theo:main-finite} 
Let us have a random field of obstacles as described in Assumption~\ref{ass:1}.
There exist a constant $ C = C( n , s , r_{0} , r_{1} , \lambda , f_{0} ) $
such that for any $ U \in C^{2}( \R^{n} ) $ of the form 
\[ U(x) = \nu \cdot x + r(x) \]
with $ \nu \in \R^{n} $ and $ r : \R^{n} \to \R $ sublinear it holds:
If  $ \| \nabla U \|_{\i} < 1 $ and 
\[ \| ( - \Laplace )^{s} U \|_{\i} 
\le C ( 1 - \| \nabla U \|_{\i} )^{ 1 + \frac{n}{2s} }, \]
then there exists $ F^{*} > \| ( - \Laplace )^{s} U \|_{\i} $ such that 
there exists a.s.~a viscosity supersolution
$ v : \R^{n} \times \Omega \to \R $ to
\[ - ( - \Laplace )^{ s } v( x , \omega ) - f( x , u( x , \omega ) , \omega ) + F^{*} \le 0 
\quad \mbox{and} \quad
v( x , \omega ) > U(x). \]
\end{theo}
\item
We may additionally assess the distance of the supersolution to the initial function.
For the smallest $y$ in Theorem~\ref{theo:percolation}, it holds
\[ {\mathbf E}( y_{0} ) 
= \sum_{ m = 0 }^{\i} \mathbf{P}( y_{0} > m )
\le \sum_{ m = 0 }^{\i} C_{n,p,H} \frac{ 2^{m} (1-p)^{m} }{ 2p-1 }
= \frac{ C_{n,p,H} }{ ( 2p-1 )^{2} }
< \i. \]
(One could explicitly determine the constant $ C_{n,p,H} $.)
Therefore, there is a supersolution that additionaly for every $ x \in \R^{n} $ fulfils
\[ {\mathbf E}( v(x,\prostor) - U(x) ) \le M < \i \] 
with $ M = M( n , s , r_{0} , r_{1} , \lambda , f_{0} ) $.
\end{trivlist}
\section{Homogenization} \label{sect:hom}
Theorem~\ref{theo:main-finite} immediately yields a homogenization result for to the physically most interesting case $ s = \frac{1}{2} $.
Let us first explain the setting. 
We suppose that for every $ \eps > 0 $, we have a random field of obstacles that scales with $ \eps $.
More precisely, let the random function of the obstacle field be
\[ f^{\eps}( x , y , \omega ) := f( \tfrac{x}{\eps} , \tfrac{y}{\eps} , \omega ) \]
where $f$ is as in Assumption~\ref{ass:1}. 
For each $ \eps > 0 $, we explore the behaviour of same the interface given by a function $ u_{0} $ at time $ t = 0 $
that is determined by
\begin{eqnarray*}
\partial_{t} u^{\eps}(t,x, \omega ) & = & - (- \Delta )^{1/2} u^{\eps}(t,x, \omega ) - f^{\eps}(x, u^{\eps}(t,x), \omega ) + F, \\
u^{\eps}(0,x, \omega ) & = & u_{0}(x).
\end{eqnarray*} 
We wish to determine what happens for $ \eps \to 0 $. 
We look at the same random field from more and more large distance where, however,
the initial interface remains the same.
%
Clearly, for every $ \eps $ the obstacles have the maximal strength in cube of side $ 2 \eps r_{0} $ 
and are correspondingly more tightly distributed, 
namely, still according to a Poisson point process but now with parameter $ \frac{\lambda}{\eps^{n}} $.
%
%
Let us for each $ \eps > 0 $ consider existence of a viscosity supersolution $ v^{\eps} $  to
\[ - ( - \Laplace )^{1/2} v^{\eps}(x) - f^{\eps}( x , v^{\eps}(x) ) + F \le 0  
\quad \mbox{and} \quad
v^{\eps} > u_{0} \]
where $ u_{0} \in C^{2}( \R^{2} ) $ fulfils the assumptions of Theorem~\ref{theo:main-finite}. 
Let us rescale 
\[ v^{\eps}(x) =: \eps w^{\eps}( \tfrac{x}{ \eps } ). \]
Since
\[ \nabla v^{\eps}(x) = ( \nabla w^{\eps} ) ( \tfrac{x}{ \eps } )
\quad \mbox{and} \quad 
( - \Laplace )^{1/2} v^{\eps}(x) = ( ( - \Laplace )^{1/2} w^{\eps} ) ( \tfrac{x}{ \eps } ), \]
the new function $ w^{\eps} $ must fulfil for all $ x \in \R^{n} $
\[ - ( ( - \Laplace )^{1/2} w^{\eps} )(  \tfrac{x}{ \eps } ) 
- f^{\eps}( x , \eps w^{\eps}(  \tfrac{x}{ \eps } ) ) + F \le 0,  \]
i.e.~for every $ \xi \in \R^{n} $
\[ - ( - \Laplace )^{1/2} w^{\eps}( \xi ) - f( \xi , w^{\eps}( \xi ) ) + F \le 0  \]
with $ w^{\eps}( \xi ) > \frac{1}{\eps} u_{0} ( \eps \xi ) =: u_{0}^{\eps}( \xi ) $.
This is now the same inequality as in the previous sections.
Since 
\[ \| \nabla u_{0}^{\eps} \|_{\i} = \| \nabla u_{0} \|_{\i} 
\quad \mbox{and} \quad 
\| ( - \Laplace )^{1/2} u_{0}^{\eps} \|_{\i} = \| ( - \Laplace )^{1/2} u_{0} \|_{\i}, \]
we find by Theorem~\ref{theo:main-finite} for every scale $ \eps > 0 $ such a supersolution $ w^{\eps} $. 
Moreover, we may achieve
\[ {\mathbf E}( w^{\eps}(x,\prostor) - u_{0}^{\eps}(x) ) \le M < \i \] 
for all $ x \in \R^{n} $.
Hence, for every $ \eps > 0 $ there is a supersolution such that $ v^{\eps} > u_{0} $ and 
\[ {\mathbf E}( v^{\eps}( x , \prostor) - u_{0}(x) ) 
= {\mathbf E}( \eps w^{\eps}( \tfrac{x}{ \eps } , \prostor ) - \eps u_{0}(  \tfrac{x}{ \eps }  ) )
\le \eps M. \] 
Since the upper blocking interface converges in expectation towards the initial interface as $ \eps \to 0 $,
so does the solution. 
\begin{theo} \label{theo:homogenization}
Let us have a random field of obstacles fulfilling Assumptions~\ref{ass:1}.
There exist a constant $ C = C( n , s , r_{0} , r_{1} , \lambda , f_{0} ) $ 
and a force $ F^{*} = F^{*}( n , s , r_{0} , r_{1} , \lambda , f_{0} ) $
such that for any bounded $ u_{0} \in C^{2}( \R^{n} ) $ with
\begin{equation*}
\| \nabla u_{0} \|_{\i} \le C 
\quad \mbox{and} \quad 
\| ( - \Delta )^{s} u_{0} \|_{\i} \le C , 
\end{equation*} 
and for any $ F \le F^{*} $ the viscosity solutions
$ u^{\eps} : \R \times \R^{n} \times \Omega \to \R $ to
\begin{eqnarray*}
\partial_{t} u(t,x, \omega ) & = & - (- \Delta )^{1/2} u(t,x, \omega ) - f^{\eps}(x, u(t,x), \omega ) + F, \\
u(0,x, \omega ) & = & u_{0}(x),
\end{eqnarray*} 
converge to the initial value, i.e.~for all $ x \in {\mathbb R}^{n} $ and $ t \in {\mathbb R} $
\[ {\mathbf E}[ ( u^{\varepsilon}(t,x, \prostor ) - u_{0}(x) )_{+} ]
\le \eps M \to 0
\quad 
\mbox{as } \eps \to 0, \]
where $ (.)_{+} $ denotes the positive part.
\end{theo}
We stated the theorem above with the positive part 
since the initial data may induce in some parts of the interface motion downwards. We note that the physical situation would be that the obstacles require an additional force for the interface to pass over them independently of the direction, see, e.g.,~\cite{KCO}. This would correspond to an additional $L^1$-type dissipation located at the obstacle sites. However, for such models no proven comparison principles exist so far.

In our simpler model, we introduce only obstacles that exert  a downward force on the interface. For a bound on the interface from below, we would therefore be allowed  to assume that the obstacles lying below the interface act upwards.
The analogous analysis can be performed for $ - u $,  yielding also a bound for the negative part of same type.


\bibliography{refs}
\bibliographystyle{alphaabbr}

\end{document}